%% file: main.tex
\documentclass[12pt]{article}
\usepackage{amsmath,amssymb,amstext,dsfont,fancyvrb,float,fontenc,graphicx,subfigure,ntheorem,hyperref}
\usepackage[nameinlink]{cleveref}
\usepackage[utf8]{inputenc}

\usepackage[numbers, sort&compress]{natbib}
\usepackage{tikz,everypage}
\usepackage[letterpaper]{geometry}
\ifx\volno\undefined\def\volno{0}\fi
\ifx\volyear\undefined\def\volyear{2017}\fi
\ifx\pagno\undefined\def\pagno{000--000}\fi

\newfont{\footsc}{cmcsc10 at 8truept}
\newfont{\footbf}{cmbx10 at 8truept}
\newfont{\footrm}{cmr10 at 10truept}

\usepackage{fancyhdr}
\usepackage{relsize}
\usepackage{sectsty}
\allsectionsfont{\larger[-1]} 

\renewcommand\paragraph{\@startsection{paragraph}{4}{\z@}
                                    {2ex \@plus.5ex \@minus.2ex}
                                    {-1em}
                                    {\normalfont\normalsize\bfseries}}

\renewcommand\subparagraph{\@startsection{subparagraph}{5}{\parindent}
                                       {2ex \@plus.5ex \@minus .2ex}
                                       {-1em}
                                      {\normalfont\normalsize\bfseries}}

\newlength{\BiblioSpacing}
\renewenvironment{thebibliography}[1]{
\begin{oldthebibliography}{#1}
\setlength{\parskip}{\BiblioSpacing}
\setlength{\itemsep}{\BiblioSpacing}
}
{
\end{oldthebibliography}
}

\usepackage[strict]{changepage}
\def\abstractname{Abstract -}   % <-----------------
\def\abstract{\begin{adjustwidth}{1cm}{1cm} \par    \footnotesize \noindent {\bf \abstractname} 
\def\endabstract{ \end{adjustwidth} \smallskip }}
{\theorembodyfont{\itshape}\newtheorem{theorem}{Theorem}[section]}
{\theorembodyfont{\itshape}\newtheorem{proposition}[theorem]{Proposition}}
{\theorembodyfont{\itshape}\newtheorem{definition}[theorem]{Definition}}
{\theorembodyfont{\itshape}\newtheorem{lemma}[theorem]{Lemma}}
{\theorembodyfont{\itshape}\newtheorem{corollary}[theorem]{Corollary}}
{\theorembodyfont{\rm}}
{\theorembodyfont{\rm}}
{\theorembodyfont{\rm}}
{\theorembodyfont{\rm }\newtheorem{example}[theorem]{Example}}

\usepackage{import}

\renewcommand{\phi}{\varphi} 
\newcommand{\set}[1]{\left\{#1\right\}}

\usetikzlibrary{math}
\usetikzlibrary{matrix}
\usetikzlibrary{calc}
\usetikzlibrary{arrows.meta}
\usetikzlibrary{decorations.pathreplacing,calligraphy}

\newcommand{\org}{|[fill=orange!50]|}
\newcommand{\blue}{|[fill=blue!25]|}
\newcommand{\red}{|[fill=red!50]|}

\hypersetup{
    colorlinks=true,       % false: boxed links; true: colored links
    linkcolor=blue,        % color of internal links (change box color with linkbordercolor)
    citecolor=blue,        % color of links to bibliography
}

\title{\Large\bf Antimagic Labelings of Forests
\thanks{This work was supported by an NSF  grant for undergraduate research, NSF DMS 1916494.}}
\author{\sc Johnny Sierra, Daphne Der-Fen Liu, and Jessica Toy}
\begin{document}
\setcounter{page}{1}
%\date{}
%%%%%%%%%%%%%%%%%%%%%%%%%%%%%%%%%%%%%%%%%
\maketitle
\thispagestyle{fancy}

\vskip 1.5em

\begin{abstract}
An antimagic labeling of a graph $G(V,E)$ is a bijection $f: E \to \{1,2, \dots, |E|\}$ so that $\sum_{e \in E(u)} f(e)  \neq \sum_{e \in E(v)} f(e)$ holds for all $u, v \in V(G)$ with $u \neq v$, where $E(v)$ is the set of edges incident to $v$.  We call $G$ antimagic if it admits an antimagic labeling. A forest is a graph without cycles; equivalently, every component of a forest is a tree. 
It was proved by Kaplan, Lev, and Roditty [2009], and by Liang, Wong, and Zhu [2014] that every tree with at most one vertex of degree-2 is antimagic. A major tool used in the proof is the zero-sum partition introduced  by Kaplan, Lev, and Roditty [2009]. In this article, we provide an algorithmic representation for the zero-sum partition method and apply this method to show that every forest with at most one vertex of degree-2 is also antimagic. 
\end{abstract}
 

\begin{MSC}
05C78; 05C05
\end{MSC}
\section{Introduction} 

The notion of magic graphs was motivated by magic squares. A magic square is an $n\times n$ array of integers $\{1,2,\dots, n^2\}$ so that each row, column, and the main and back-main diagonals of the square sum to the same value (see \Cref{fig:magic} for an example). In 1963, Sedl\'{a}\v{c}ek \cite{first} extended this concept to graphs by labeling the edges of a graph with numbers and defining the \textit{vertex-sum} of each vertex to be the total of the labels  assigned to the edges incident to that vertex. \Cref{fig:magic} shows how a magic square is used to create a \textit{magic  labeling} on a complete bipartite graph, where all the vertex-sums are identical.

An \textit{edge labeling} of a graph $G$ is a function $f$ that assigns to each edge of $G$ a positive integer. For a vertex $u \in V(G)$, denote $E(u)$ the set of edges incident to $u$.  The \textit{vertex-sum} of $u$ is defined as

\begin{equation}
\label{defn}
\phi_f(u)=\sum_{e\in E(u)} f(e). 
\end{equation} 

\begin{figure}[t]
    \centering
    \include{magic.tex}
    \vspace{-12.5mm}
    \caption{A magic square along with a magic $K_{3,3}$ graph. Each row and column of the magic square sum to 15. Each vertex in $K_{3,3}$ has vertex-sum 15.}
    \label{fig:magic}
\end{figure}

An edge labeling $f$ is called {\it magic} if all vertices have the same vertex-sum (see \Cref{fig:magic}). Many variations of magic labelings have been studied (cf. \cite{Baca2019}). Among them, the {\it antimagic labeling} has been studied widely in recent decades. The definition is given below. For positive integers $a \leq b$, denote $[a,b]=\{a, a+1, \dots, b\}$.  
\begin{definition}
\label{def:antimagic}
Let $G=(V,E)$ be a graph with $m$ edges. A bijective function  $f:  E \to [1,m]$  is an {\bf antimagic labeling} for $G$ if $\phi_f(u)\neq \phi_f(v)$ for any two vertices $u\neq v$, where the vertex-sum $\phi_f(u)$ is defined in \Cref{defn}.  If $G$ admits such an antimagic labeling, then $G$ is said to be {\it antimagic}. 
\end{definition}

By \Cref{def:antimagic}, it is clear that $K_2$, the simple graph with only one edge and two vertices, is not antimagic. The notion of antimagic labelings was introduced by Hartsfield and Ringel in \cite{hartsfield2013pearls} in 1990, who conjectured that every connected graph other than $K_2$  is antimagic. Since then, this conjecture has been studied extensively. Many families of graphs are known to be antimagic, yet the conjecture remains open (cf. \cite{dense, Baca2019, regular2, regular, regular bipartite, caterpillars, consecutive}).  

A {\it tree} is a connected graph without cycles. For a graph $G$ and a vertex $v \in V(G)$, the {\it degree} of $v$, denoted by $\deg(v)$, is the number of edges incident to $v$. The following result was due to Kaplan, Lev, and Roditty \cite{KAPLAN20092010},  and Liang, Wong, and  Zhu \cite{LIANG20149}:  

\begin{theorem}[\cite{KAPLAN20092010, LIANG20149}]
\label{thm} 
A tree $T \neq K_2$ with at most one degree-2 vertex  is antimagic.
\end{theorem}

\noindent
A major tool used in proving \Cref{thm} is called the {\it zero-sum partition}. It was first devised by Kaplan, Lev, and Roditty in \cite{KAPLAN20092010}. In Section 2, we provide a step-by-step process of this  method. 

A {\it forest} is a graph without cycles. Thus,  every component of a forest is a tree, called a {\it component tree}. We aim to investigate antimagic labelings of forests by utilizing the zero-sum partition method of antimagic labelings for trees. The main result of this article is: 

\begin{theorem}
\label{thm:newresult}
A forest $F$ with at most one degree-2 vertex and without  $K_2$ component trees is antimagic.
\end{theorem}

The proof of \Cref{thm:newresult} is presented in Section 3. In Section 2, we introduce the zero-sum partition method that will be used in the proof of \Cref{thm:newresult}. In Section 4, we discuss possible directions and open problems for future study. 

\section{The Zero-Sum Partition Method}
The zero-sum partition method is based on the following two results. Here we present proofs that provide a step-by-step partition algorithm. 

\begin{lemma}
[\cite{KAPLAN20092010}]
\label{lem:zerosum}
Let $s,l$ be non-negative integers and let $k=2s+6l$. Then there is a partition of $[1, k]$ into subsets $Q_1, Q_2, \dots, Q_{s+2l}$ such that the following hold: 
\[
\begin{array}{llll} 
\mbox{For $i \in [1, l]$:}    &|Q_i|=3 \  \mbox{and} \ \sum\limits_{a\in Q_i}a=k+1. \\    
\mbox{For $i \in [1+l, s+l]$:} &|Q_i|=2 \ \mbox{and} \  \sum\limits_{a\in Q_i}a=k+1. \\
\mbox{For $i \in [s+l+1, s+2l]$:} &|Q_i|=3 \ \mbox{and} \ \sum\limits_{a\in Q_i}a=2(k+1).
\end{array}
\hspace{2in}
\]
\end{lemma}

\begin{proof} 
The proof in \cite{KAPLAN20092010} provides formulas of sets $Q_i$ directly. To assist in understanding, we provide a step-by-step method to obtain these formulas. 
In general, \Cref{lem:zerosum} partitions the set of numbers in $[1, k]$ ($k=2s+6l$)  into three types of sets, called  $A$-, $B$-, and $C$-sets, where each $A$-set has three elements with sum $k+1$, each $B$-set has two elements with sum $k+1$, and each $C$-set has three elements with sum $2(k+1)$. 
Precisely, there are $l$ $A$-sets (denoted as  $A_1, A_2, \dots, A_l$), $s$ $B$-sets (denoted as $B_{1}, B_{2}, \dots, B_{s}$), and $l$ $C$-sets (denoted as $C_{1}, C_{2}, \dots, C_{l}$). 

The strategy of getting these $A$-, $B$-, and $C$-sets is expressed  in the following four steps. Along the process, we use the following \Cref{example} to illustrate each step. 

\begin{example}
\label{example}
Suppose $s=5$ and $l=2$. Then $s+3l=11$ and $k=2(s+3l) =22$.  
\noindent
By \Cref{lem:zerosum}, we can partition the numbers in $[1, 22]$ into subsets $A_1,A_2,B_1, \dots ,B_5, C_1, C_2$.
\end{example}

\medskip
\noindent
{\bf Step 1: Arranging all the labels into a two-row matrix.} \ List the numbers in $[1,k]$ as a $2 \times (\frac{k}{2})$ matrix $M$ where the first row has numbers $1, 2, \dots, \frac{k}{2}$ in increasing order, and the second row has numbers $k, k-1, \dots, \frac{k}{2} + 1$ in decreasing order.  Precisely, $a_{1,i}=i$ and $a_{2,i} = k-i+1$, for $i \in[1,\frac{k}{2}]$. See \Cref{fig:step1} as an illustration for \Cref{example}. Consequently, the two numbers in each column sum up to $k+1$. Note that there are  $\frac{k}{2}=s+3l$ columns. 

\begin{figure}[h]
    \centering
    \include{step2}    
    \vspace*{-12.5mm}
    \caption{Arrange $[1,22]$ into a two-row matrix $M$. Group the columns as shown. In total, there are $s+3l$ columns.}
    \label{fig:step1}
\end{figure}

\medskip
\noindent
{\bf Step 2: Determining the B-sets.} \  
Fix the columns of $[l+1, l+s]$ in $M$ as the $B$-sets. Precisely,  
$$
\mbox{{\bf B-sets}}: B_i = \{a_{1, l+i}, \ a_{2,l+i}\} = \{l+i, \ k-l-i+1\},  \ i \in [1, s]. 
$$
Note that the sum of each $B$-set is $k+1$.
See \Cref{fig:step2} for an illustration of Step 2 for \Cref{example}. 

\begin{figure}[h]
    \centering
    \include{step3a}
    \vspace*{-12.5mm}
    \caption{Use these $s$ columns to create the $B$-sets. The $B$-sets are $B_1 = \{3,20\}, B_2=\{4,19\},  B_3=\{5,18\}, B_4=\{6,17\}$, and $B_5=\{7,16\}$.}
    \label{fig:step2}
\end{figure}

\medskip
\noindent
{\bf Step 3: Determining the $A$-sets.} \  We use the first $l$ elements on the first row of $M$ to create A-sets. Recall that each column $i \in [1,l]$ in $M$, the sum of the two elements is $k+1$, since $a_{1,i}+a_{2,i}=(i)+(k-i+1)=k+1$. 
We replace $a_{2,i}$ by two elements, $a_{1, l+s+i}$ and $a_{2, l+s+2i}$, since: 
$$
a_{1, l+s+i} +  a_{2, l+s+2i} = l+s+i + [ k-(l+s+2i)+1]= k-i+1=a_{2,i}.
$$
Hence, we obtain the $l$ $A$-sets as follows:  
\[
\mbox{{\bf $A$-sets}}: A_i =\set{a_{1,i}, \ a_{1,s+l+i}, \   a_{2,s+l+2i}}=\{i, \ l+s+i, \  k-(l+s+2i)+1\} ,\  i \in [1,l]. 
\]
See \Cref{fig:step3} for an illustration of this step for \Cref{example}. 

\begin{figure}[h]
    \centering
    \include{step4a}
    \vspace*{-12.5mm}
    \caption{The orange boxes are used to create $A$-sets. The process is to exchange the bottom numbers from the first $l$ columns with numbers in the last $2l$ columns. Notice that, 22=8+14 and 21 = 9+12. The $A$-sets are $A_1=\{1,8,14\}$ and $A_2 = \{2, 9, 12\}$.}
    \label{fig:step3}
\end{figure}

After Steps 1, 2, and 3, the remaining unused labels are 
\[
\begin{array}{lll}
a_{1, i}, &i \in [s+2l+1, s+3l] \\
a_{2, i}, &i \in [1,l] \cup \{s+l+1, s+l+3, \dots, s+3l-1\}. 
\end{array}
\]

\noindent
{\bf Step 4: Determining the $C$-sets.} \   
For $i \in [1, l]$, first fix $a_{2, i} = k+1-i \in C_{i}$. 
Next, we combine the unused labels 
$a_{2, s+3l-2i}$ and $a_{1, s+3l-i}$ to form $C_i$: 
$$
\mbox{{\bf C-sets}}: 
C_{i} = \{a_{2,i}, a_{1, s+3l+1-i}, \ a_{2, s+3l+1-2i}\} = \{k-i+1, \ s+3l+1-i, \ k-(s+3l+1-2i)+1\}. 
$$
One can easily check that $C_i$ is a $C$-set since the sum of the three elements is $2(k+1)$. See \Cref{fig:step4} for an illustration of Step 4 and the final partition for \Cref{example}. 

\begin{figure}[h]
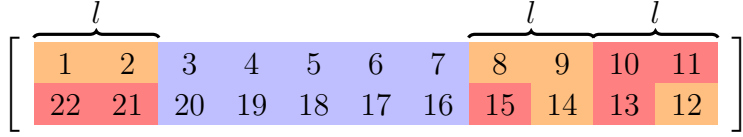

    \centering
    \include{cstep}
    \vspace*{-12.5mm}
    \caption{The remaining numbers are used to create $C$-sets. The $C$-sets are $C_1=\{22,11,13\}$ and $C_2=\{21,10,15\}$.}
    \label{fig:step4}
\end{figure}

After the four steps, each number in $[1,k]$ belongs to exactly one of the $A$-, $B$-, $C$-sets. 
It is clear that the numbers in the first row and the numbers in the second row up to $a_{2, l+s}$ are used exactly once. For the remaining numbers, $a_{2, i}$, $i \in [s+l+1, s+3l]$, exactly half are in the $A$-sets while the other half are in the  $C$-sets, according to the parity of $i$.  See \Cref{fig:covering} for an illustration. 
\end{proof}

\begin{figure}[h]
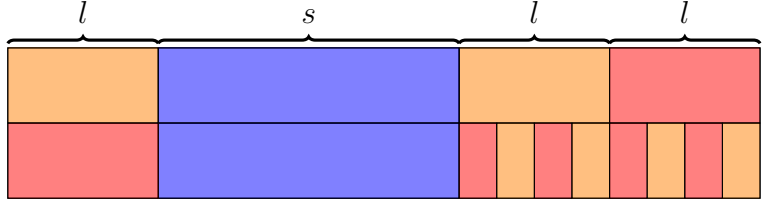

    \centering
    \include{covering}
    \vspace*{-12.5mm}
    \caption{This figure helps visualize how $[1,k]$ is partitioned. $A$-sets cover the blocks in orange, $B$-sets cover the blocks in blue, and $C$-sets cover the blocks in red.}
    \label{fig:covering}
\end{figure}

 The next corollary is an immediate consequence of \Cref{lem:zerosum}, which allows us to partition the set of  integers $[1,k]$ into subsets of any sizes greater than one.    
 
\begin{corollary}[\cite{KAPLAN20092010,LIANG20149}]
\label{cor:part}
Let $k=r_1+r_2+\dotsb+r_t$ be a partition of the positive integer $k$, where $r_i\geq 2$ for $i \in [1,t]$. 
Then the set $[1,k]$ can be partitioned into pairwise disjoint subsets, $D_1, D_2, \dots, D_t$, such that for every $1\leq i \leq t$, $|D_i|=r_i$ and $\sum_{a\in D_i} a \equiv 0 \pmod{k'}$, where $k'=k+1$ if $k$ is even, and $k'=k$ if $k$ is odd.
\end{corollary}

\begin{proof} Let $k=r_1+r_2+\dotsb+r_t$, where each $r_i\geq 2$ for $i=1,2,\dots,t$. First, assume $k$ is even. Write each $r_i$ as the sum of multiples of 2 and multiples of 3, $r_i=2m_i+3n_i$, $m_i, n_i \geq 0$. Then $k = \sum_{i=1}^t (2m_i + 3n_i)$. Because $k$ is even, $n_1+n_2 +\dots+n_t$ must be even. Denote
$$s=m_1+m_2+\dotsb+m_t \ \ \mbox{and} \ \    l=\frac{n_1+n_2+\dotsb +n_t}{2}.
$$
Then, $k=2s+6l$. By \Cref{lem:zerosum},
$[1,k]$ can be partitioned into sets $Q_1$, $Q_2,\dots, Q_{s+2l}$. 

We denote this partition of $[1,k]$ as ${\cal Q}$. To create the set $D_i$ with $|D_i|=r_i=2m_i+3n_i$, take $m_i$ sets in ${\cal Q}$ that are 2-element sets,  and take $n_i$ sets in ${\cal Q}$ that are 3-element sets. With the selected sets in ${\cal Q}$, let $D_i$ be the union of those sets. Since each 2- or 3-element set in ${\cal Q}$ has the elements sum to $0  \pmod{k+1}$, it implies that the elements in $D_i$ also sum to $0  \pmod{k+1}$. 

Now assume $k=r_1+r_2 + \dots + r_t$ is odd. Then there exists some odd $r_j$, $r_j \geq 3$. Then $k-1=r_1+r_2+\dotsb (r_j-1)+\dotsb +r_t$. Note that $r_j -1 \geq 2$. As $k-1$ is even, following the procedure described above, we can partition the set $[1,k-1]$ into the union of sets $D_i$, $i \in [1, t]$, where $|D_i|=r_i$ for $i \neq j$ and $|D_j|=r_j-1$, and the sum of each set $D_i$ is a multiple of $k$, $i \in [1,t]$. Once the sets are created, 
replace $D_j$ by   $\{k\}\cup D_{j}$ and keep the other $D_i$, $i \neq j$, to create a partition for $[1,k]$ which satisfies the statement.
\end{proof}

The following definitions show how we apply  \Cref{cor:part} to obtain an antimagic labeling for some trees. A {\it leaf} of a graph is a degree-1 vertex; a non-leaf vertex is called an {\it internal vertex}.

\begin{definition}
\label{def:root_trees}
An {\bf (oriented) rooted tree} is a tree where one vertex is designated to be the root. We draw the root at the top of the tree and orient edges from top to down.  Define the root to be at level 0. A vertex $v$ is at level-$i$ if the distance from $v$ to the root is $i$.  All the edges of the tree are oriented from vertices in level-$i$ to vertices in level-$(i+1)$. For a vertex $v$, the {\bf children of $v$} are the vertices that are adjacent to $v$ and are one level higher than $v$. We call $v$ the parent of its children. We say the tree is oriented from parent to children. 
\end{definition}

\begin{definition}
\label{def:edges}
Let $T$ be a rooted tree. For every non-root vertex $v$, we denote $e^v$ as the {\bf incoming edge} (the edge from the parent) of $v$. This is well-defined since each vertex (other than the root) has exactly one incoming edge. 
\end{definition}

\begin{definition}
Let $T$ be a rooted tree with $m$ edges. For $v \in V(T)$, denote  $E^+(v)$ the set of {\bf outgoing edges}  from $v$ and denote $|E^+(v)|=n_v$. 
\end{definition}

Let  $v_1, v_2, \dots,  v_t$ be the vertices of $T$ with $n_{v_i} \geq 1$. It readily follows that 
\begin{equation}
    m=|E(T)|=n_{v_1} + n_{v_2} + \dots + n_{v_t}.
\label{equ1}
\end{equation}

\begin{definition}
For a rooted tree $T$ with $m$ edges, we call a bijective mapping  $f:E(T)\to [1,m]$ a {\bf zero-sum labeling} of $T$ if for each $i \in [1,t]$, 
$$
\sum_{e \in E^+(v_i)} f(e) \equiv 0 \  \mbox{(mod $m'$)},
$$
where $m'=m$ if $m$ is odd; and $m'=m+1$ if $m$ is even. 
\end{definition}

\begin{proposition}
{\rm \cite{KAPLAN20092010, LIANG20149}}
\label{prop}
Let $T$ be a tree with an even number of edges and at most one vertex of degree-2. 
Then there exists a zero-sum labeling for $T$ that is antimagic. 
\end{proposition}

\begin{proof}
Let $T$ be a tree with $m$ edges, where $m$ is even. Root $T$ at the degree-2 vertex if $T$ has a degree-2 vertex; otherwise, root $T$ at any internal vertex. Denote this root as $w$. Let $v_1, v_2, \dots, v_t$ be the internal vertices of $T$ and $n_{i}$ the number of outgoing edges of $v_i$. By our assumption, for $i\in[1,t]$, $n_i \geq 2$. By \Cref{equ1} and \Cref{cor:part}, there exists a zero-sum labeling $f$ for $T$ by assigning numbers in $D_i$ to the outgoing edges of $v_i,\  i \in[1,t]$.  Further, for every non-root vertex $v \in V(T)$, $\phi_f(v) = f(e^{v}) + \sum_{e \in E^+(v)} f(e) \equiv f(e^{v}) \pmod{m+1}$ (if $v$ is a leaf then $E^+(v)=\emptyset$). For the root vertex $w$, $\phi_f(w) \equiv 0 \not\equiv f(e^{v}) \pmod{m+1}$ for any $v$. Since each $v$ has a distinct incoming edge, no two internal vertices will have congruent sums$\pmod{m+1}$. Thus, all vertices have distinct vertex sums, implying $f$ is an antimagic labeling for $T$.
\end{proof}

\begin{figure}
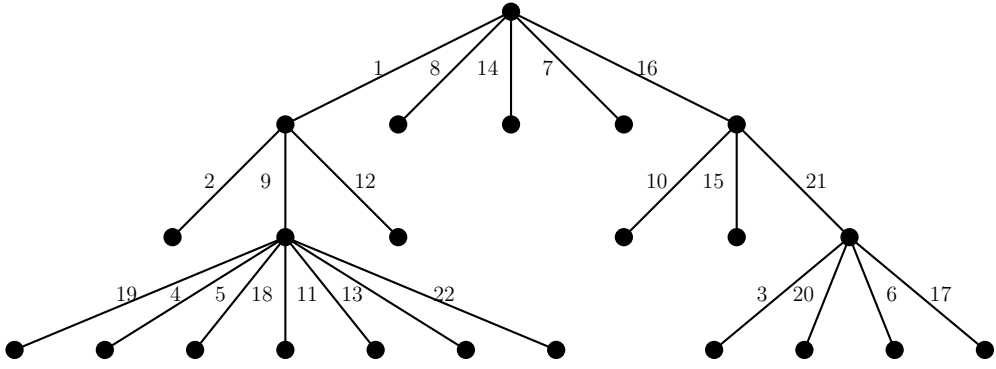

    \centering
    \include{antimagic_tree_example}
    \vspace*{-12.5mm}
    \caption{A zero-sum antimagic labeling of a tree using the partition of \Cref{example} in \Cref{fig:step4}.} 
    \label{fig:compatLabel}
\end{figure}

See \Cref{fig:compatLabel} for an illustration of a zero-sum labeling of a tree by using \Cref{cor:part}. In the proof of \Cref{thm} in \cite{KAPLAN20092010,LIANG20149} the tree is carefully rooted so that a zero-sum labeling is also an antimagic labeling. Following this idea, in the proof of \Cref{thm:newresult} (next section), we carefully root each component tree of a given forest $F$ at a vertex and then identify all these roots as a  single vertex to form a single tree $T$. We find a zero-sum labeling for $T$ using the zero-sum partition method. Then we use this labeling to produce an antimagic labeling for $F$. 

%%%%%%%%%%%%%
% 
\section{Proof of \texorpdfstring{\Cref{thm:newresult}}{Theorem 2}}
%
%%%%%%%%%%%%%

Let $F$ be a forest with $s$ component trees and $m=|E(F)|$. Denote the component trees of $F$ by $T_i$ with root at $w_i$ for  $i \in[1,s]$. If $s=1$, then $F$ is a tree, and the result follows by \Cref{thm}. Henceforth we assume $s\geq 2$. We proceed with the proof by considering cases. 

\noindent 
{\bf Case 1: $m$ is odd.} Consider two sub-cases.

\medskip

{\bf Sub-case 1.1: $F$ has no degree-2 vertex.} Let $w_i\in V(T_i)$ be a leaf of $T_i$ for $i \in [1, s]$. Root each $T_i$ at $w_i$ and orient from parents to children as defined in \Cref{def:root_trees}. Denote $w_i'$ to be the child of $w_i$ and denote the edge between them as $e_i = w_iw_i'$.  

Let $T$ be the tree obtained by identifying the vertices $w_1, w_2, \dots, w_s$ into a single vertex $w$, where $w$ is the root of $T$. Let $v_1, v_2, \dots, v_t$ be the non-leaf vertices of $T$. Assume each vertex $v_i$ has $n_i$ outgoing edges, $n_i \geq 2$. Recall from \Cref{equ1}, the number of edges of $T$ is  
$m = n_1 + n_2 + \dots + n_t$. Since $m$ is odd, there exists some odd $n_j$ so that $n_j \geq 3$.
By  \Cref{cor:part} one can partition $[1,m]$ into sets $D_1,  D_2,  \dots, D_t$, where $|D_i|=n_i \geq 2$, and the elements in each  $D_i$ sum up to a multiple of $m$. Further, the label $m$ is assigned to an outgoing edge $v_jv_{j'}$ of $v_j$ with $n_j \geq 3$.   This gives a zero-sum labeling $f$ for $T$, where $\phi_f(v) \equiv  f(e^{v})$ (mod $m$) holds for all vertex, except the root $w$.  Thus all values of $\phi_f(v_i)$, $i \in [1,t]$, are distinct, except that $\phi_f(w) \equiv \phi_f(v_{j'}) \equiv 0$ (mod $m$), as $f(v_jv_{j'})=m$.   

Let $g$ be the labeling for $F$ defined by 
\[
g(e) = \begin{cases}
f(ww_i') & \textrm{if } e=w_iw_i',\\
f(e) & \textrm{otherwise.}
\end{cases}
\]

Note that $v_{j'}$  is the only vertex with $\phi_g(v_{j'})=\phi_f(v_{j'}) \equiv 0$ (mod $m$).  For every non-root vertex $u$ of $F$ (i.e., $u \neq w_i$ for all $i \in [1,s]$), we have $\phi_g(u) = \phi_f(u) \equiv f(e^u) \pmod{m}$. Thus $\phi_g(u)\pmod{m}$ are pair-wisely distinct. If $w_i$ and $w_j$ are two root vertices of the component trees, $i \neq j$, then $\phi_g(w_i)=f(e_i) \neq f(e_j) = \phi_g(w_j)$. Thus the roots have different vertex-sums. Further, $\phi_g(w_i)\equiv \phi_g(w_i') \equiv f(e_i) \pmod{m}$. Since $w_i'$ has at least two children, we have $\phi_g(w_i') \geq f(e_i)+ m > f(e_i) = \phi_g(w_i)$, implying $\phi_g(w_i)\neq \phi_g(w_i')$. 
Hence, $g$ is an antimagic labeling for $F$. See \Cref{fig:case1.1} as an example. 

\begin{figure}[H]
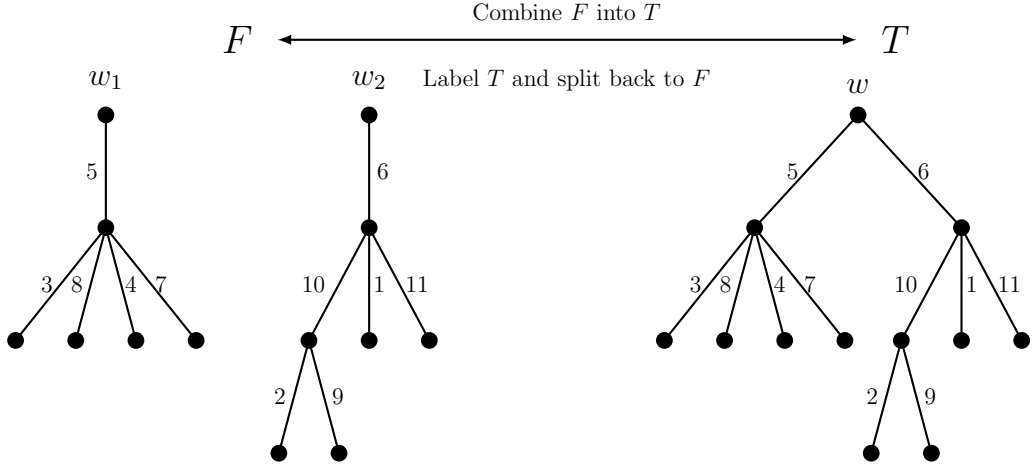

    \centering
    \include{odd_nodeg2}
    \vspace*{-12.5mm}
    \caption{An example for Sub-case 1.1. The forest $F$ to the left has two component trees, $T_1$ and $T_2$, and $m=11$ edges. Root $T_1$ and $T_2$ at a leaf, $w_1$ and $w_2$,  and identify $w_1$ and $w_2$ into a single root $w$ to form a tree $T$ shown on the right. A zero-sum labeling on $T$ gives an antimagic labeling on $F$ when $T$ is split back into the components of $F$.} 
    \label{fig:case1.1}
\end{figure}

{\bf Sub-case 1.2: $F$ contains a degree-2 vertex.} Suppose $F$ has a degree-2 vertex, $v'$. Without loss of generality, assume $v' \in V(T_1)$. Root each $T_i$ at a leaf $w_i\in V(T_i)$ for $i \in [1,s]$. 
Let $w_i'$ be the child of $w_i$ and let $v''$ be the (only) child of $v'$. Let $T$ be the tree obtained by identifying $w_1, w_2, \dots, w_s$ into a vertex $w$, which is the  root of $T$. Note that it is possible that $v'$ is the child of $w_1$ in $T_1$.  Let $w=v_1, v_2, \dots, v_t$ be the internal vertices of $T \setminus \{v'\}$.    
By \Cref{cor:part}, we can partition the numbers in the set $[1,m-1]$ into sets that sum up to multiples of $m$ and assign these sets to the outgoing edges of $v_i$, $i \in [1,t]$. Finally,  we assign $m$ to the edge $v'v''$.  This labeling, denoted by $f$, is a zero-sum labeling for $T$, where $\phi_f(u) \equiv f(e^u) \pmod{m}$ holds for every non-root vertex $u$ and $f(v'v'')=m$. Note that $\phi_f(v') \equiv f(e^{v'}) \pmod{m}$.
Let $g$ be the labeling for $F$ defined as follows:
\[
g(e)=
\begin{cases}
    f(ww_i') & \textrm{ if } e=w_iw_i',\\
    f(e) & \textrm{ otherwise.}
\end{cases}
\]
Then $\phi_g(u)\equiv g(e^u) \pmod{m}$ if $u \neq w_i$. Note that $v''$ is the only vertex with $\phi_g(v'')\equiv 0 \pmod{m}$. If $u\neq v$, then $\phi_g(u)\not\equiv \phi_g(v) \pmod{m}$ since $u$ and $v$ have different incoming edge labels. However, we have $\phi_g(w_i)\equiv \phi_g(w_i') \pmod{m}$.  Each $w_i'$ has at  least two children or one child if $w_i'=v'$ (the degree-2 vertex), so $\phi_g(w_i')\geq g(w_iw_i')+m > g(w_iw_i') = \phi_g(w_i)$. Thus $g$ is an antimagic labeling for $F$. See \Cref{fig:case1.2} as an example. 

\begin{figure}[h]
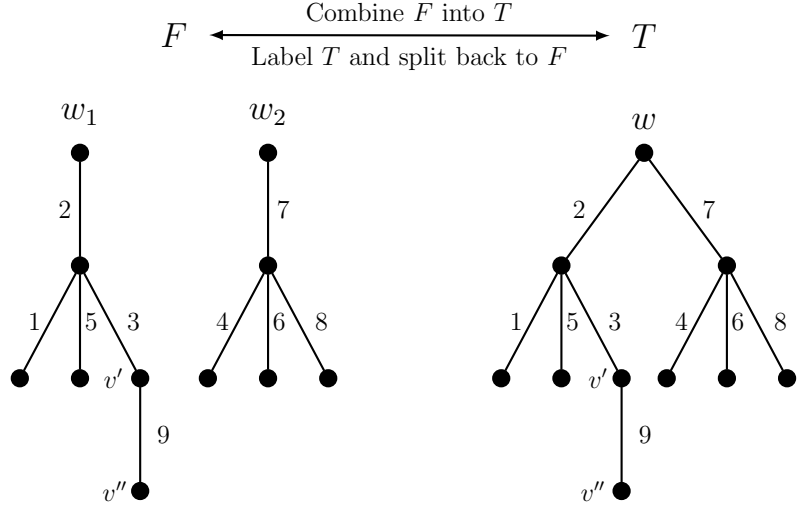

    \centering
    \include{odd_deg2_forest_and_tree}      
    \vspace*{-12.5mm}
    \caption{An example for Sub-case 1.2. The forest $F$ on the left consists of two trees $T_1$ and $T_2$, and $m=9$ edges, where $T_1$ has a degree-2 vertex $v'$ and roots of $T_1$ and $T_2$  are leaves. We identify these two roots to create a tree $T$ (on the right) and find a  zero-sum labeling on $F$ using \Cref{cor:part} and labeling $v'v''$ with $m$. Afterwards, split $T$ back into $T_1$ and $T_2$ to get an antimagic labeling for $F$ as shown.} 
    \label{fig:case1.2}
\end{figure}

\medskip

\noindent 
{\bf Case 2: $m$ is even.}

\medskip

{\bf Sub-case 2.1: $F$ has no degree-2 vertex.}  For $i \in [1, s]$, root $T_i$ at a leaf $w_i$.  
Let $w_i'$ be the child of $w_i$. Let $T$ be the tree obtained by identifying the vertices $w_1, w_2, \dots, w_s$ into a single vertex $w$. Denote $v_1, v_2, \dots, v_t$ the non-leaf vertices of $T$. Assume each vertex $v_i$ has $r_i$ children. Then $|E(T)|=m = r_1 + r_2 + \dots + r_t$, where $r_i \geq 2$ for $i \in [1,t]$. 
By \Cref{prop} there is a zero-sum labeling $f$ for $T$, which is antimagic. 

Let $g$ be the labeling for $F$ defined by 
\[
g(e) = \begin{cases}
f(ww_i') & \textrm{for } e=w_iw_i', \ i \in[1,s];\\
f(e) & \textrm{otherwise.}
\end{cases}
\]
Similar to Case 1, one can easily show that $g$ is an antimagic labeling of $F$.

\medskip

{\bf Sub-case 2.2. $F$ contains a degree-2 vertex.} 
Assume $F$ has a degree-2 vertex $u$. Without loss of generality, assume $u$ is located in $T_1$. Consider two possibilities. 
 
\medskip

{\bf $\diamond$ Sub-case 2.2.1.}  $s=2$.  Root $T_1$ at $w_1=u$. Root $T_2$ at  $w_2$ where $\deg_{T_2}({w_2}) \geq 3$. By \Cref{cor:part} there is a zero-sum labeling $f$ for $F$. In this labeling, we assign a $B$-set to edges of $E^+(w_1)$. The vertex-sums are distinct in modulo $m+1$, excluding the two roots $w_1$ and $w_2$.  
Since $w_1$ is degree-2, 
 $\phi_f(w_1)=m+1$. Further, as deg$(w_2) \geq 3$, we can choose the labels for edges incident to $w_2$ to sum up to at least $2(m+1)$ (if deg$(w_1)=3$, we use a $C$-set from \Cref{lem:zerosum}).  Thus the vertex-sums are all distinct. Therefore $f$ is an antimagic labeling for $F$. See \Cref{fig:case2.2}.

\medskip

{\bf $\diamond$ Sub-case 2.2.2.} $s \geq 3$.  Let $w_1=u$ be the root of $T_1$, and $E^+(u)=\{e', e''\}$.  For the remaining component trees $T_i$, $i \in [2, s]$, root $T_i$ at a leaf $w_i \in V(T_i)$. Let $T$ be the tree obtained by identifying $w_1,w_2,\dots, w_s$ into a single vertex $w$, and root $T$ at $w$.
By \Cref{cor:part}, there exists a zero-sum labeling for $T$ such that $f(e')$ and $f(e'')$ belong to the same $B$-set in the partition, $f(e')+ f(e'')=m+1$. This is possible as $s \geq 3$, and so deg$(w) \geq 4$.  

Let $g$ be the labeling for $F$ defined by
\[
g(e) = \begin{cases}
f(wv) & e=w_iv, \ i \in [1,s];\\
f(e) & \textrm{otherwise.}
\end{cases}
\] 

\noindent 
Then $\phi_g(w_1) = m+1 \equiv 0 \pmod{m+1}$. 
Similarly to the above cases, one can show that $g$ is an antimagic labeling for $F$.
See \Cref{fig:case2.2.2} as an example. 
This completes the proof of \Cref{thm:newresult}. 
\hfill$\Box$

\begin{figure}[H]
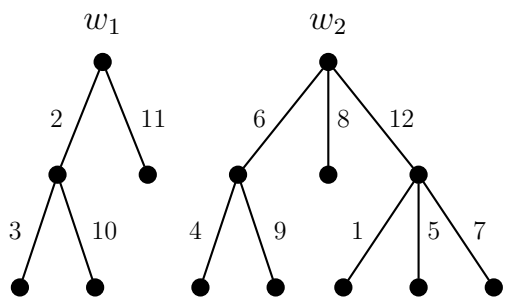

    \centering
    \include{case_2-2-1-v2}
    \vspace*{-12mm}
    \caption{An example of Sub-case 2.2.1. This forest has two components, $T_1$ and $T_2$, $m=12$ edges, and a degree-2 vertex $w_1$. Root $T_1$ at $w_1$, and root $T_2$ at a vertex $w_2$ of degree-3 or higher. By \Cref{cor:part} there exists a zero-sum antimagic labeling so that $\phi(w_1)=m+1$ and $\phi(w_2) \geq 2(m+1)$.}  
    \label{fig:case2.2}
\end{figure}

\begin{figure}[H]
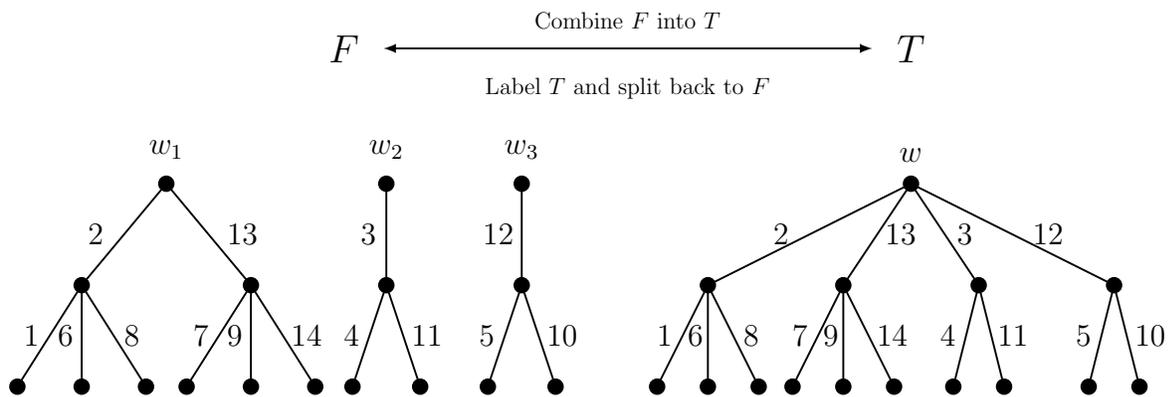

    \include{case_2-2-2}
    \vspace*{-12mm}
    \caption{An example of Sub-case 2.2.2. The forest $F$ has 3 components, $T_1,T_2,$ and $T_3$. Root $T_1$ at the degree-2 vertex $w_1$, and root the other trees at a leaf. When labeling the combined tree $T$, the edges incident to $w_1$  get a $B$-set, so that it is the only vertex with sum 0 (mod 15).}
    \label{fig:case2.2.2}
\end{figure}

\section{Future Study}

In this article, we proved that every forest without $K_2$ as a component tree and having at most one vertex of degree-2 is antimagic. It would be interesting to investigate the forests where each component tree contains at most one degree-2 vertex.  

Suppose $G$ is a graph and $e=uv$ is an edge of $G$. A subdivision of $e$ is the operation of replacing $e=uv$ with a path $(u,w_e,v)$, where $w_e$ is a new vertex.
For a tree $T$, 
denote by $T^*$ the tree obtained by subdividing each edge in $T$.
The following result was proved in \cite{LIANG20149}. 

\begin{theorem} {\rm (\cite{LIANG20149})} 
If $T$ is a tree without degree-2 vertices, then $T^*$ is antimagic. 
\end{theorem}

Define $F^*$ in a similar way to $T^*$.
Suppose $F$ is a forest. 
Denote by $F^*$ the forest obtained by subdividing each edge in $F$. 

\medskip

\noindent\textbf{Question.} For any forest $F$ without a degree-2 vertex and without $K_2$ as a component, is it true that $F^*$ is antimagic?

\section*{Acknowledgments} %%% ADD HERE YOUR THANKS AND ACKNOWLEDGMENTS; please do not change spelling
We are grateful for the support of NSF PUMP (Preparing Undergraduates through Mentoring towards Ph.D.'s) grant. We would like to thank the anonymous referee for careful reading of the  manuscript with helpful comments.  

%\section*{Instructions for preparing the reference list}

%The following should be observed when preparing the list:
%\begin{itemize}
%\item leave no space between strings of initials
%\item all authors names appear, separated by commas (do not use et al.)
%\item  there is a comma after the article title
%\item the journal title is abbreviated (use the standard abbreviation, if you don't know it you can look it up on MathSciNet or zbMATH)
%\item the journal title appears in italics
%\item in the case of books, the book title appears in italics
%\item the volume number appears in bold face, and the issue number is not listed
%\item the year appears in parentheses after the volume number, and is followed by a comma
%\item use a double dash between page numbers 
%\end{itemize}

%\newpage 

%%%%%%%% AUTHORS' INFORMATION. DELETE/ADD AUTHORS AS NEEDED
{\footnotesize  
\medskip
\medskip
\vspace*{1mm} 

\noindent {\it Johnny Sierra}\\  
California State University, Los Angeles\\
5151 State University Dr.\\
Los Angeles, CA 90032 \\
E-mail: {\tt jsierr29@calstatela.edu}\\ \\  

\noindent {\it Daphne Liu}\\  
California State University, Los Angeles\\
5151 State University Dr.\\
Los Angeles, CA 90032 \\
E-mail: {\tt dliu@calstatela.edu}\\ \\

\noindent {\it Jessica Toy}\\  
California State University, Los Angeles\\
5151 State University Dr.\\
Los Angeles, CA 90032 \\
E-mail: {\tt jtoy@calstatela.edu} \\ \\   
%%%%%%%%%%%% Please do not remove or move the } sign below, do not remove blank line before it

}

%%%%% Please do not change or remove the lines below, they will be changed in copyediting

\vspace*{1mm}\noindent\footnotesize{\date{ {\bf Received}: April 31, 2023\;\;\;{\bf Accepted}: July 20, 2023}}\\
\vspace*{1mm}\noindent\footnotesize{\date{  {\bf\it PUMP Journal of  Undergraduate Research}}}

\end{document}

%% file: magic.tex
\centering
\begin{tikzpicture}
 \matrix at (-3,0)[matrix of nodes, nodes={rectangle,draw}, minimum size = 7.5mm]
  {
    2 & 7 & 6\\
    9 & 5 & 1 \\
    4 & 3 & 8 \\
  };
\foreach \x in {0,2,4}
\foreach \y in {0,2,4}
{\draw (\y,-1) -- (\x,1);}
\foreach \x in {0,2,4}{
\draw[fill=black] (\x,-1) circle (2pt);
\draw[fill=black](\x,1) circle (2pt);}
% draws the labels for top left vertex
\draw(0,0.5) node[left] {2};
\draw (0.5,0.5) node[left]{9};
\draw(0.8,0.7) node[right=-2mm] {4};
% draws the labels for top middle vertex
\draw (1.65, 0.7) node[left]{7};
\draw(2,0.5) node[right]{5};
\draw (2.4,0.7) node[right]{3};
% draws the labels for the top right vertex
\draw (3.7,0.9) node[left]{6};
\draw (3.5,0.5) node[right]{1};
\draw (4,0.5) node[right]{8};
\end{tikzpicture}

%% file: step2.tex
\begin{tikzpicture}
    \matrix  at (0,-2.75)
    [  
        matrix of math nodes, 
        nodes in empty cells,
        minimum width=width("8888"),
        left delimiter={[},right delimiter={]}
    ] (B)
    { 
        1  &  2 &  3 &  4 &  5 &  6 &  7 &  8 &  9 & 10 & 11\\
        22 & 21 & 20 & 19 & 18 & 17 & 16 & 15 & 14 & 13 & 12\\
    };

    % This draws the sideway braces 
    \draw[decorate, decoration = {brace,raise=1.7pt},very thick]
    (B-1-1.north west) -- (B-1-2.north east) node[midway,above=2.75pt] {$l$}; % first l col.
    
    \draw[decorate, decoration = {brace,raise=1.7pt},very thick]
    (B-1-3.north west) -- (B-1-7.north east) node[midway,above=2.75pt] {$s$}; % next s col.
    
    \draw[decorate, decoration = {brace,raise=1.7pt},very thick]
    (B-1-8.north west) -- (B-1-9.north east) node[midway,above=2.75pt] {$l$}; % next l col.
    
    \draw[decorate, decoration = {brace,raise=1.7pt},very thick]
    (B-1-10.north west) -- (B-1-11.north east) node[midway,above=2.75pt] {$l$}; % next l col.
\end{tikzpicture}

%% file: step3a.tex
\begin{tikzpicture}
    % Step 3A: Creating the B-Sets
    \matrix  at (0,-5.25)
    [  
        matrix of math nodes, 
        nodes in empty cells,
        minimum width=width("8888"),
        left delimiter={[},right delimiter={]},
    ] (C)
    { 
        1  &  2 &  \blue3 &  \blue4 &  \blue5 &  \blue6 &  \blue7 &  8 &  9 & 10 & 11\\
        22 & 21 & \blue20 & \blue19 & \blue18 & \blue17 & \blue16 & 15 & 14 & 13 & 12\\
    };
    
    \draw[decorate, decoration = {brace,raise=1.7pt},very thick]
    (C-1-3.north west) -- (C-1-7.north east) node[midway,above=2.75pt] {$s$}; % next s col.
\end{tikzpicture}

%% file: step4a.tex
\begin{tikzpicture}
        \matrix  at (0,-7.75)
    [  
        matrix of math nodes, 
        nodes in empty cells,
        minimum width=width("8888"),
        left delimiter={[},right delimiter={]},
    ] (D)
    { 
        \org1  & \org 2 &  \blue3 &  \blue4 &  \blue5 &  \blue6 &  \blue7 &  \org8 &  \org9 & 10 & 11\\
        22 & 21 & \blue20 & \blue19 & \blue18 & \blue17 & \blue16 & 15 & \org14 & 13 & \org12\\
    };
    
    %\draw[decorate, decoration = {brace,raise=1.7pt},very thick]
    %(D-1-3.north west) -- (D-1-7.north east) node[midway,above=2.75pt] {$s$}; % next s col.
    \draw[decorate, decoration = {brace,raise=1.7pt},very thick]
    (D-1-1.north west) -- (D-1-2.north east) node[midway,above=2.75pt] {$l$}; % first l col.
    
    \draw[decorate, decoration = {brace,raise=1.7pt},very thick]
    (D-1-8.north west) -- (D-1-9.north east) node[midway,above=2.75pt] {$l$}; % next l col.
    
    \draw[decorate, decoration = {brace,raise=1.7pt},very thick]
    (D-1-10.north west) -- (D-1-11.north east) node[midway,above=2.75pt] {$l$}; % next l col.

    %\draw[fill=blue, opacity=0.25] (D-1-3.north west) rectangle (D-2-7.south east);
\end{tikzpicture}

%% file: cstep.tex
\begin{tikzpicture}
    \matrix  at (0,-10.25)
    [  
        matrix of math nodes, 
        nodes in empty cells,
        minimum width=width("8888"),
        left delimiter={[},right delimiter={]},
    ] (D)
    { 
        \org1  & \org 2 &  \blue3 &  \blue4 &  \blue5 &  \blue6 &  \blue7 &  \org8 &  \org9 & \red10 & \red11\\
        \red22 & \red21 &\blue 20 & \blue19 & \blue18 &\blue 17 & \blue16 & \red15 & \org14 & \red13 & \org12\\
    };
    
    %\draw[decorate, decoration = {brace,raise=1.7pt},very thick]
    %(D-1-3.north west) -- (D-1-7.north east) node[midway,above=2.75pt] {$s$}; % next s col.
    \draw[decorate, decoration = {brace,raise=1.7pt},very thick]
    (D-1-1.north west) -- (D-1-2.north east) node[midway,above=2.75pt] {$l$}; % first l col.

    \draw[decorate, decoration = {brace,raise=1.7pt},very thick]
    (D-1-8.north west) -- (D-1-9.north east) node[midway,above=2.75pt] {$l$}; % next l col.
    
    \draw[decorate, decoration = {brace,raise=1.7pt},very thick]
    (D-1-10.north west) -- (D-1-11.north east) node[midway,above=2.75pt] {$l$}; % next l col.
    
    %\draw[fill=blue, opacity=0.25] (D-1-3.north west) rectangle (D-2-7.south east);
\end{tikzpicture}

%% file: covering.tex
\begin{tikzpicture}
    \draw (0,0) rectangle (10,2); %draws the rectangle
    \draw (0,1) node{} -- (10,1) node{}; %divides the row

    \draw (2,0) node{} -- (2,2)  node{}; %first l columns
    \draw (6,0) node{} -- (6,2)  node{}; %next s columns
    \draw (8,0) node{} -- (8,2)  node{}; %next l columns
    
    \draw (1,2.1) node[above=2.5pt,black]{$l$}; %first l columns label
    \draw [decorate,
    decoration = {brace,raise=1.7pt},very thick] (0,2) -- (2,2); % l brace
    
    \draw (4,2.1) node[above=2.5pt,black]{$s$}; %next s col label
    \draw [decorate,
    decoration = {brace,raise=1.7pt},very thick] (2,2) -- (6,2);% s brace
    
    \draw (7,2.1) node[above=2.5pt,black]{$l$}; %next l col label
    \draw [decorate,
    decoration = {brace,raise=1.7pt},very thick] (6,2) -- (8,2); % next l brace
    
    \draw (9,2.1) node[above=2.5pt,black]{$l$}; %next l col label
    \draw [decorate,
    decoration = {brace,raise=1.7pt},very thick] (8,2) -- (10,2); % next l brace
    
    \draw[fill=orange, fill opacity=0.5] (0,1) rectangle (2,2); %fills the A-set columns (first l)
    \draw[fill=orange, fill opacity=0.5] (6,1) rectangle (8,2); %fills the A-set columns (second l)
    
    \draw[fill=blue, fill opacity=0.5] (2,1) rectangle (6,2); %fills the B-set columns
    \draw[fill=blue, fill opacity=0.5] (2,0) rectangle (6,1); %fills the B-set columns
    
    \draw[fill=red, fill opacity=0.5] (0,0) rectangle (2,1); %fills the A-set columns (first l)
    \draw[fill=red, fill opacity=0.5] (8,1) rectangle (10,2); %fills the A-set columns (second l)
    
    %alternating colors 
    \draw[fill=red, fill opacity=0.5] (6,0) rectangle (6.5,1);
    \draw[fill=orange, fill opacity=0.5] (6.5,0) rectangle (7,1);
    \draw[fill=red, fill opacity=0.5] (7,0) rectangle (7.5,1);
    \draw[fill=orange, fill opacity=0.5] (7.5,0) rectangle (8,1);
    \draw[fill=red, fill opacity=0.5] (8,0) rectangle (8.5,1);
    \draw[fill=orange, fill opacity=0.5] (8.5,0) rectangle (9,1);
    \draw[fill=red, fill opacity=0.5] (9,0) rectangle (9.5,1);
    \draw[fill=orange, fill opacity=0.5] (9.5,0) rectangle (10,1);
\end{tikzpicture}

%% file: antimagic_tree_example.tex
\begin{tikzpicture}[nodes={draw,circle, fill=black, scale=0.55}, -, thick]
\tikzstyle{level 3}=[sibling distance=12mm]
\node[]{}
    child{node{}
        child{node{}
        edge from parent node[draw=none,fill=none,left,scale=1.25]{2}
        }
        child{node{}
            child{node{}
            edge from parent node[draw=none,fill=none,left,scale=1.25]{19}
            }
            child{node{}
            edge from parent node[draw=none,fill=none,left,scale=1.25]{4}
            }
            child{node{}
            edge from parent node[draw=none,fill=none,left,scale=1.25]{5}
            }
            child{node{}
            edge from parent node[draw=none,fill=none,left,scale=1.25]{18}
            }
            child{node{}
            edge from parent node[draw=none,fill=none,left,scale=1.25]{11}
            }
            child{node{}
            edge from parent node[draw=none,fill=none,left,scale=1.25]{13}
            }
            child{node{}
            edge from parent node[draw=none,fill=none,right,scale=1.25]{22}
            }
        edge from parent node[draw=none,fill=none,left,scale=1.25]{9}
        }
        child{node{}
        edge from parent node[draw=none,fill=none,right,scale=1.25]{12}
        }
    edge from parent node[draw=none,fill=none,left,scale=1.25]{1}
    }
    child{node{}
    edge from parent node[draw=none,fill=none,left,scale=1.25]{8}
    }
    child{node{}
    edge from parent node[draw=none,fill=none,left,scale=1.25]{14}
    }
    child{node{}
    edge from parent node[draw=none,fill=none,left,scale=1.25]{7}
    }
    child{node{}
        child{node{}
        edge from parent node[draw=none,fill=none,left,scale=1.25]{10}
        }
        child{node{}
        edge from parent node[draw=none,fill=none,left,scale=1.25]{15}
        }
        child{node{}
            child{node{}
            edge from parent node[draw=none,fill=none,left,scale=1.25]{3}
            }
            child{node{}
            edge from parent node[draw=none,fill=none,left,scale=1.25]{20}
            }
            child{node{}
            edge from parent node[draw=none,fill=none,right,scale=1.25]{6}
            }
            child{node{}
            edge from parent node[draw=none,fill=none,right,scale=1.25]{17}
            }
        edge from parent node[draw=none,fill=none,right,scale=1.25]{21}
        }
    edge from parent node[draw=none,fill=none,right,scale=1.25]{16}    
    };
\end{tikzpicture}

%% file: odd_nodeg2.tex
\begin{tikzpicture}[nodes={draw,circle, fill=black,scale=0.5}, -, thick]
    \tikzstyle{level 1}=[sibling distance = 2.75cm]
    \tikzstyle{level 2}=[sibling distance = 0.8cm]
    \node[above,draw=none,fill=none,scale=2] at  (-6,0) {\(w_1\)};
    \node at (-6,0) {}
        child{ node{}
            child{ node{}
            edge from parent node[draw=none,fill=none,left=-2mm,scale=1.5]{3}
            }
            child{ node{}
            edge from parent node[draw=none,fill=none,left=-2mm,scale=1.5]{8}
            }
            child{ node{}
            edge from parent node[draw=none,fill=none,right=-3mm,scale=1.5]{4}
            }
            child{ node{}
            edge from parent node[draw=none,fill=none,right=-3mm,scale=1.5]{7}
            }
        edge from parent node[draw=none,fill=none,left=-2mm,scale=1.5]{5}
        };

    \node[above,draw=none,fill=none,scale=2] at  (-2.5,0) {\(w_2\)};
    \node at (-2.5,0) {}
        child{ node{}
            child{ node{}
                child{ node{}
                edge from parent node[draw=none,fill=none,left=-2mm,scale=1.5]{2}
                }
                child{ node{}
                edge from parent node[draw=none,fill=none,right=-2mm,scale=1.5]{9}
                }
            edge from parent node[draw=none,fill=none,left,scale=1.5]{10}
            }
            child{ node{}
            edge from parent node[draw=none,fill=none,right=-3mm,scale=1.5]{1}
            }
            child{ node{}
            edge from parent node[draw=none,fill=none,right=-2mm,scale=1.5]{11}
            }
        edge from parent node[draw=none,fill=none,right=-2mm,scale=1.5]{6}
        };

    \node[above,draw=none,fill=none,scale=2] at  (4,0) {\(w\)};
    \node at (4,0) [fill] {}
        child{ node{}
            child{ node{}
            edge from parent node[draw=none,fill=none,left=-2mm,scale=1.5]{3}
            }
            child{ node{}
            edge from parent node[draw=none,fill=none,left=-2mm,scale=1.5]{8}
            }
            child{ node{}
            edge from parent node[draw=none,fill=none,right=-3mm,scale=1.5]{4}
            }
            child{ node{}
            edge from parent node[draw=none,fill=none,right=-3mm,scale=1.5]{7}
            }
        edge from parent node[draw=none,fill=none,left=-2mm,scale=1.5]{5}
        }
        child{ node{}
            child{ node{}
                child{ node{}
                edge from parent node[draw=none,fill=none,left=-2mm,scale=1.5]{2}
                }
                child{ node{}
                edge from parent node[draw=none,fill=none,right=-2mm,scale=1.5]{9}
                }
            edge from parent node[draw=none,fill=none,left,scale=1.5]{10}
            }
            child{ node{}
            edge from parent node[draw=none,fill=none,right=-3mm,scale=1.5]{1}
            }
            child{ node{}
            edge from parent node[draw=none,fill=none,right=-2mm,scale=1.5]{11}
            }
        edge from parent node[draw=none,fill=none,right=-2mm,scale=1.5]{6}
        };
    %\draw[help lines] (-5,-5) grid (5,5);
    
    \node [draw=none,fill=none,scale=2.5] (F) at (-4.25,1) {\(F\)};
    \node [draw=none,fill=none,scale=2.5] (T) at (4.5,1) {\(T\)};
    \draw [latex-latex] (F) -- (T);

    \node[draw=none,fill=none,above=-20mm,scale=1.5] at ($(F)!0.5!(T)$) {Combine \(F\) into \(T\) };
    \node[draw=none,fill=none,below=-30mm,scale=1.5] at ($(F)!0.5!(T)$) {Label \(T\) and split back to \(F\)};
\end{tikzpicture}

%% file: odd_deg2_forest_and_tree.tex
    \begin{tikzpicture}[nodes={draw,circle, fill=black,scale=0.55}, -, thick]
    \tikzstyle{level 1}=[sibling distance=25mm]
    \tikzstyle{level 2}=[sibling distance=22mm]
    \tikzstyle{level 3}=[sibling distance=8mm]
    \tikzstyle{level 4}=[sibling distance=10mm]
    \tikzset{deg2/.style={draw,circle, fill=red, scale=0.75 }} 
    \node[white]{}
    child{ node(A){}
        child{ node{}
            child{ node{}
            edge from parent node[draw=none,fill=none,left=-2mm,scale=1.5]{1}
            }
            child{ node{}
            edge from parent node[draw=none,fill=none,right=-3mm,scale=1.5]{5}
            }
            child{ node{}
                child{ node{} node[draw=none, fill=none,scale=1.5,left=-.5mm]{$v''$}
                edge from parent node[draw=none,fill=none,right,scale=1.5]{9}
                }
             node[draw=none, fill=none,scale=1.5,left]{$v'$}
             edge from parent node[draw=none,fill=none,right,scale=1.5]{3}
            }
        edge from parent node[draw=none,fill=none,left=-2mm,scale=1.5]{2}
        }
    edge from parent[draw=none]
    }
    child{ node(B){}
        child{node{}
            child{node{}
            edge from parent node[draw=none,fill=none,left=-2mm,scale=1.5]{4}
            }
            child{node{}
            edge from parent node[draw=none,fill=none,right=-3mm,scale=1.5]{6}
            }
            child{node{}
            edge from parent node[draw=none,fill=none,right,scale=1.5]{8}
            }
        edge from parent node[draw=none,fill=none,right=-2mm,scale=1.5]{7}
        }
    edge from parent[draw=none]
    }
    child[sibling distance=50mm]{node(w){}
        child{ node{}
        child{ node{}  edge from parent node[draw=none,fill=none,left=-2mm,scale=1.5]{1}}
        child{ node{}  edge from parent node[draw=none,fill=none,right=-3mm,scale=1.5]{5}}
        child{ node{} 
            child{ node{}  node[draw=none, fill=none,scale=1.5,left]{$v''$} edge from parent node[draw=none,fill=none,right,scale=1.5]{9}}
        node[draw=none, fill=none,scale=1.5,left=-.5mm]{$v'$}
        edge from parent node[draw=none,fill=none,right,scale=1.5]{3}
        }
    edge from parent node[draw=none,fill=none,left,scale=1.5]{2}
    }
    child{ node{} 
        child{ node{}  edge from parent node[draw=none,fill=none,left=-2mm,scale=1.5]{4}}
        child{ node{}  edge from parent node[draw=none,fill=none,right=-3mm,scale=1.5]{6}}
        child{ node{}  edge from parent node[draw=none,fill=none,right,scale=1.5]{8}}
    edge from parent node[draw=none,fill=none,right,scale=1.5]{7}
    }
    edge from parent [draw=none]
    };
    \node [above, draw=none, fill=none,scale=2] at (A) {$w_1$};
    \node [above, draw=none, fill=none,scale=2] at (B) {$w_2$};
    \node [above, draw=none, fill=none, scale=2] at (w)
    {$w$};
    \node[draw=none,fill=none,scale=2,above=10mm] (F) at ($(A)!0.5!(B)$) {$F$};
    \node[draw=none,fill=none,scale=2,above=10mm](T) at (w){\(T\)};
    \draw[latex-latex] (F) -- (T);
    \node[draw=none,fill=none,above=-22mm,scale=1.5] at ($(F)!0.5!(T)$) {Combine \(F\) into \(T\) };
    \node[draw=none,fill=none,below=-35mm,scale=1.5] at ($(F)!0.5!(T)$) {Label \(T\) and split back to \(F\)};
    \end{tikzpicture}

%% file: case_2-2-1-v2.tex
\begin{tikzpicture}[nodes={draw,circle, fill=black, scale=0.55}, -, thick]
\tikzstyle{level 1}=[sibling distance=12mm]
\tikzstyle{level 2}=[sibling distance=10mm]
\tikzstyle{level 3}=[sibling distance=10mm]
\tikzstyle{level 4}=[sibling distance=10mm]
\tikzset{deg2/.style={draw,circle, fill=red, scale=0.75 }}
    \node[draw=none,fill=none,above,scale=2] at (-5,0) {$w_1$};
    \node at (-5,0) {}
        child{ node{}
            child{ node{}
            edge from parent node[draw=none,fill=none,left,scale=1.5]{3}
            }
            child{ node{} 
            edge from parent node[draw=none,fill=none,right,scale=1.5]{10}
            }
        edge from parent node[draw=none,fill=none,left,scale=1.5]{2}
        }
        child{ node{}
        edge from parent node[draw=none,fill=none,right,scale=1.5]{11}
        };
    \node[draw=none,fill=none,above,scale=2] at (-2,0) {$w_2$};
    \node at (-2,0) {}
        child{ node{}
            child{ node{}
            edge from parent node[draw=none,fill=none,left,scale=1.5]{4}
            }
            child{ node{}
            edge from parent node[draw=none,fill=none,right,scale=1.5]{9}
            }
        edge from parent node[draw=none,fill=none,left,scale=1.5]{6}
        }
        child{ node{}
        edge from parent node[draw=none,fill=none,right=-2mm,scale=1.5]{8}
        }
        child{ node{}
            child{ node{}
            edge from parent node[draw=none,fill=none,left,scale=1.5]{1}
            }
            child{ node{}
            edge from parent node[draw=none,fill=none,right=-2mm,scale=1.5]{5}
            }
            child{ node{}
            edge from parent node[draw=none,fill=none,right,scale=1.5]{7}
            }
        edge from parent node[draw=none,fill=none,right,scale=1.5]{12}
        };
\end{tikzpicture}

%% file: case_2-2-2.tex
\scalebox{0.9}{
\begin{tikzpicture}[nodes={draw,circle, fill=black, scale=0.55}, -, thick]
\tikzstyle{level 1}=[sibling distance=25mm]
\tikzstyle{level 2}=[sibling distance=9.5mm]

\node[draw=none,fill=none,above,scale=2] at (-2,0) {$w_1$};

\node at (-2,0) {}
    child{ node{}
        child{ node{} 
        edge from parent node[draw=none,fill=none,left=-2.75mm,scale=2] {1}
        }
        child{ node{}
        edge from parent node[draw=none,fill=none,left=-3.25mm,scale=2] {6}
        }
        child{ node{}
        edge from parent node[draw=none,fill=none,right=-2.75mm,scale=2] {8}
        }
    edge from parent node[draw=none,fill=none,left,scale=2] {2}
    } 
    child{ node{}
        child{ node{}
        edge from parent node[draw=none,fill=none,left=-2.75 mm,scale=2] {7}
        }
        child{ node{}
        edge from parent node[draw=none,fill=none,left=-3.25mm,scale=2] {9}
        }
        child{ node{}
        edge from parent node[draw=none,fill=none,right=-2.75mm,scale=2] {14}
        }
    edge from parent node[draw=none,fill=none,right,scale=2] {13}
    };

\tikzstyle{level 1}=[sibling distance=30mm]
\tikzstyle{level 2}=[sibling distance=10mm]

\node[draw=none,fill=none,above,scale=2] at (1.25,0) {$w_2$};
\node at (1.25,0) {}
    child{ node{}
        child{ node{}
        edge from parent node[draw=none,fill=none,left=-2.75mm,scale=2] {4}
        }
        child{ node{}
        edge from parent node[draw=none,fill=none,right=-2.7mm,scale=2] {11}
        }
    edge from parent node[draw=none,fill=none,left=-2.75mm,scale=2] {3}
    };

\node[draw=none,fill=none,above,scale=2] at (3.25,0) {$w_3$};
\node at (3.25,0) {}
    child{ node{}
        child{ node{}
        edge from parent node[draw=none,fill=none,left=-2.75mm,scale=2] {5}
        }
        child{ node{}
        edge from parent node[draw=none,fill=none,right=-2.75mm,scale=2] {10}
        }
    edge from parent node[draw=none,fill=none,left=-2.75mm,scale=2] {12}
    };

    \tikzstyle{level 1}=[sibling distance=20mm]
    \tikzstyle{level 2}=[sibling distance=7.5mm]
    \node[draw=none,fill=none,above,scale=2] at (9,0) {$w$};
    \node at (9,0) {}
    child{ node{}
        child{ node{} 
        edge from parent node[draw=none,fill=none,left=-2.75mm,scale=2] {1}
        }
        child{ node{}
        edge from parent node[draw=none,fill=none,left=-4.25mm,scale=2] {6}
        }
        child{ node{}
        edge from parent node[draw=none,fill=none,right=-2.75mm,scale=2] {8}
        }
    edge from parent node[draw=none,fill=none,left,scale=2] {2}
    } 
    child{ node{}
        child{ node{}
        edge from parent node[draw=none,fill=none,left=-2.75 mm,scale=2] {7}
        }
        child{ node{}
        edge from parent node[draw=none,fill=none,left=-4.25mm,scale=2] {9}
        }
        child{ node{}
        edge from parent node[draw=none,fill=none,right=-2.75mm,scale=2] {14}
        }
    edge from parent node[draw=none,fill=none,right=-2.75mm,scale=2] {13}
    }
    child{ node{}
        child{ node{}
        edge from parent node[draw=none,fill=none,left=-2.75mm,scale=2] {4}
        }
        child{ node{}
        edge from parent node[draw=none,fill=none,right=-3.27mm,scale=2] {11}
        }
    edge from parent node[draw=none,fill=none,right=-2.25mm,scale=2] {3}
    }
    child{ node{}
        child{ node{}
        edge from parent node[draw=none,fill=none,left=-2.75mm,scale=2] {5}
        }
        child{ node{}
        edge from parent node[draw=none,fill=none,right=-2.75mm,scale=2] {10}
        }
    edge from parent node[draw=none,fill=none,right=+0.5mm,scale=2] {12}
    };

    \node [draw=none,fill=none,scale=2.5] (F) at (0.625,2) {\(F\)};
    \node [draw=none,fill=none,scale=2.5] (T) at (9,2) {\(T\)};
    \draw [latex-latex] (F) -- (T);
    \node[draw=none,fill=none,above=-20mm,scale=1.5] at ($(F)!0.5!(T)$) {Combine \(F\) into \(T\) };
    \node[draw=none,fill=none,below=-30mm,scale=1.5] at ($(F)!0.5!(T)$) {Label \(T\) and split back to \(F\)};
\end{tikzpicture}
}